\documentclass[12pt, reqno]{amsart}

\usepackage[usenames,dvipsnames]{xcolor}
\usepackage{amssymb, amsmath, amsthm, amsfonts, mathrsfs, mathtools}
\usepackage[backref, colorlinks = true,  linkcolor = blue, citecolor = Green]{hyperref}
\usepackage[alphabetic,backrefs,lite]{amsrefs}
\usepackage{amscd}   
\usepackage{fullpage}
\usepackage{tikz-cd} 
\usepackage[all,cmtip]{xy}
\usepackage{enumitem}
\usepackage{verbatim}

\DeclareFontEncoding{OT2}{}{} 

\newtheorem{lemma}{Lemma}[section]

\newtheorem{prop}[lemma]{Proposition}
\newtheorem{cor}[lemma]{Corollary}

\newtheorem{claim*}{Claim}
\newtheorem{thm}[lemma]{Theorem}
\newtheorem{defn}[lemma]{Definition}
\newtheorem{example}[lemma]{Example}

\theoremstyle{remark}
\newtheorem{remark}[lemma]{Remark}
\newtheorem{remarks}[lemma]{Remarks}
\newtheorem{rmk}[lemma]{Remark}

\newcommand{\A}{{\mathbb A}}

\newcommand{\PP}{{\mathbb P}}

\newcommand{\F}{{\mathbb F}}
\newcommand{\Q}{{\mathbb Q}}

\newcommand{\Z}{{\mathbb Z}}
\newcommand{\NN}{{\mathbb N}}

\newcommand{\kbar}{{\overline{k}}}

\newcommand{\kk}{{\mathbf k}}

\newcommand{\calC}{{\mathcal C}}
\newcommand{\calD}{{\mathcal D}}
\newcommand{\calE}{{\mathcal E}}
\newcommand{\calF}{{\mathcal F}}
\newcommand{\calG}{{\mathcal G}}

\newcommand{\calI}{{\mathcal I}}

\newcommand{\calM}{{\mathcal M}}
\newcommand{\calN}{{\mathcal N}}
\newcommand{\calO}{{\mathcal O}}

\usepackage[mathscr]{euscript}

\DeclareMathOperator{\Char}{char}

\DeclareMathOperator{\Hom}{Hom}

\DeclareMathOperator{\Gal}{Gal}

\DeclareMathOperator{\divv}{div}
\DeclareMathOperator{\ord}{ord}

\DeclareMathOperator{\Spec}{Spec}

\DeclareMathOperator{\red}{red}

\DeclareMathOperator{\GCD}{GCD}

\newcommand{\eps}{\varepsilon}

\numberwithin{equation}{section}
\numberwithin{table}{section}

\newcommand{\defi}[1]{\textsf{#1}} 

\newcommand{\rf}{k}

\title{Degrees of points on varieties over Henselian fields}
\author{Brendan Creutz}
\author{Bianca Viray}

\address{School of Mathematics and Statistics, University of Canterbury, Private Bag 4800, Christchurch 8140, New Zealand}
\email{brendan.creutz@canterbury.ac.nz}
\urladdr{http://www.math.canterbury.ac.nz/\~{}b.creutz}

\address{University of Washington, Department of Mathematics, Box 354350, Seattle, WA 98195,~USA}
\email{bviray@uw.edu}
\urladdr{http://math.washington.edu/\~{}bviray}

\begin{document}
\date{}
\begin{abstract}
Let $W/K$ be a nonempty scheme over the field of fractions of a Henselian local ring $R$. A result of Gabber, Liu and Lorenzini shows that the $\GCD$ of the set of degrees of closed points on $W$ (which is called the \defi{index} of $W/K$) can be computed from data pertaining only to the special fiber of a proper regular model of $W$ over $R$. We show that
the entire set of degrees of closed points on $W$ can be computed from data pertaining only to the special fiber, provided the special fiber is a strict normal crossings divisor. 

As a consequence we obtain an algorithm to compute the degree set of any smooth curve over a Henselian field with finite or algebraically closed residue field. Using this we show that degree sets of curves over such fields can be dramatically different than degree sets of curves over finitely generated fields.  For example, while the degree set of a curve over a finitely generated field contains all sufficiently large multiples of the index, there are curves over \(p\)-adic fields with index \(1\) whose degree set excludes all integers that are coprime to \(6\).
\end{abstract}
\maketitle

\section{Introduction}
	Let $W/F$ be a nonempty scheme over a field $F$. The index $\delta(W/F)$, i.e., the GCD of the degrees of the closed points of \(W\), is an important arithmetic invariant of \(W\). Clark conjectured \cite{Clark}*{Conjecture 16} that, when $F$ is the field of fractions of a Henselian local ring with perfect residue field, the index can be computed from data pertaining only to the special fiber of a proper regular model of $W$ over the ring of integers of $F$ whose special fiber is a strict normal crossings divisor (assuming such a model exists). This conjecture was proved by Gabber, Liu, and Lorenzini~\cite{GLL}, and in fact these three authors prove a stronger statement that omits the strict normal crossings hypothesis.
	
	In this paper, we prove a strengthening of Clark's conjecture in a different direction. Namely, we show how the entire \defi{degree set of \(W\),}
	\[
		\calD(W/F) := \left\{\deg_F(P) : P\in W, \textup{ closed}\right\}, \quad\textup{where }\deg_F(P):=[\kk(P):k]
	\] can be determined from data pertaining only to the special fiber of a strict normal crossings model. Since \(\delta(W/F) = \GCD(\calD(W/F))\), we obtain a simpler proof of Clark's original conjecture under the SNC hypothesis as a byproduct.

	To state our result, let us fix some notation used throughout the paper. Let $R$ be a discrete valuation ring with field of fractions $K$ and perfect residue field $k$. Let $S = \Spec(R)$ and let $s \in S$ and $\eta \in S$ denote, respectively, the closed and generic points of $S$. Let $X$ be a regular scheme and let $X \to S$ be a proper, flat morphism whose generic fiber $X_\eta$ is geometrically irreducible over $K$. The special fiber $X_s$ may be viewed as a Weil divisor $X_s = \sum_{i} m_iE_i$ on $X$, where $E_i$ are the irreducible components of $X_s$ (endowed with the reduced induced scheme structure as closed subsets of $X_s$) and $m_i \in \NN$ are their multiplicities. For a closed point $x \in X_s$ define $\calN(x) \subset \NN = \Z_{> 0}$ to be the semigroup under addition generated by the positive integers $m_i$, as $i$ ranges over the set $\{ i \;:\; x \in E_i \}$.
	
	Our main result is the following.
	
	\begin{thm}\label{thm:MainThm} Let $X \to S$ be as above. Assume that $R$ is Henselian and that $X_s$ is a strict normal crossings divisor on $X$. Then 
	\begin{equation}\label{eq:MainThm}
		\calD(X_\eta/K) = \bigcup_{x \in X_s} \deg_k(x)\calN(x)\,,
	\end{equation}
	where the union ranges over the closed points on the special fiber $X_s \subset X$.
	\end{thm}
	
	The proof shows that the degree set of the generic fiber is contained in the union appearing in~\eqref{eq:MainThm} whenever $R$ is Henselian, but that it can be a proper subset if $X_s$ does not have strict normal crossings. This containment can fail if $R$ is not Henselian (even if $X_s$ is assumed to have strict normal crossings). See Theorem~\ref{thm:general} and the examples in Section~\ref{sec:hypotheses}.
	
	From the theorem we deduce the following.
	
	\begin{cor}\label{cor:index}
		Suppose $X \to S$ satisfies the hypotheses of Theorem~\ref{thm:MainThm}. Then the index of the generic fiber of $X$ over $K$ is given by
		\[
			\delta(X_\eta/K) = \GCD( m_i\delta(E_i/k))\,.
		\]
	\end{cor}
	
The proof of Corollary~\ref{cor:index} in \cite{GLL}*{Theorem 8.2(b)} does not require the strict normal crossings hypothesis. Instead, the authors prove a moving lemma~\cite{GLL}*{Corollary 6.7} (which is of independent interest) and thereby obtain $\delta(E_i/k) = \delta(E^0_i/k)$, where $E^0_i = E_i^\textup{reg} - \cup_{i\ne j}E_i\cap E_j$. To compute the $\GCD$ it is then enough to consider the degrees of closed points on $X_\eta$ which specialize to points on the special fiber which are regular points on the reduced special fiber. This method also allows Gabber, Liu, and Lorenzini to handle cases when the irreducible components $E_i$ are not assumed to be regular (regularity of the components is implied by the strict normal crossings hypothesis). When the residue field is finite, the equality $\delta(E_i/k) = \delta(E^0_i/k)$ follows from the Weil conjectures. A proof in this case (which also works for separably closed $k$) was given earlier in \cite{BoschLiu}*{1.6} and \cite{CT-S}*{3.1}.
		
Since our interest is in computing the degree \emph{set}, rather than the subgroup of $\Z$ that it generates, we cannot consider closed points modulo rational equivalence.  Thus, we take a different approach. We characterize the Henselization of the local ring of $X$ at the closed points where $X_s$ has strict normal crossings (see Lemma~\ref{lem:SNCpoint}). This allows us to determine the possible degrees of closed points on the generic fiber which specialize to a closed point of $X_s$ where $X_s$ has strict normal crossings, whether regular on $(X_s)_{\red}$ or not.

	\subsection{Application to curves over Henselian fields}
	For any smooth projective and geometrically irreducible curve $C/K$ there exists a proper, flat morphism $X \to S$ with $X$ regular such that $X_s$ has strict normal crossings and $X_\eta \simeq C$ \cite[Proposition 10.1.8]{Liu}. Moreover, such $X \to S$ can, at least in principle, be computed given explicit (finitely presented) equations defining $C$. 
	Consequently, Theorem~\ref{thm:MainThm} yields an algorithm that reduces computation of $\calD(C/K)$ to the computation of $\calD(E^0_i/k)$ for a finite set of curves $E^0_i/k$ defined over the residue field. 
	
	Using this we show that over any Henselian field (i.e., the field of fractions of a Henselian discrete valuation ring) there is a curve of index $1$ whose degree set excludes infinitely many integers. This is of interest because over finite and Hilbertian fields (e.g., any finitely generated fields) the degree set must contain all sufficiently large multiples of the index (See Lemma~\ref{lem:finitefields}, \cite[Proposition 7.5]{GLL} and \cite[Theorem 2.1]{CMP}). For example, consider the genus $2$ hyperelliptic curve $C/K$ given by the affine equation
\begin{equation}\label{ex:2Ncup3N}
	C : y^2 = \pi(x^6 + \alpha\pi x^3 + \pi^2)\,,
\end{equation}
where $\pi \in K$ is a uniformizer and $\alpha \in R^\times$ is a unit such that $(\alpha^2 - 4) \in R^{\times 2}$. The special fiber of the minimal regular model $X$ of $C$ over $R$ is Type [IV] in the Namikawa-Ueno classification \cite[p. 155]{NU}. The irreducible components, their multiplicities and pairwise intersections (all of which are transversal) are as indicated in Diagram~\eqref{fig:1} below.

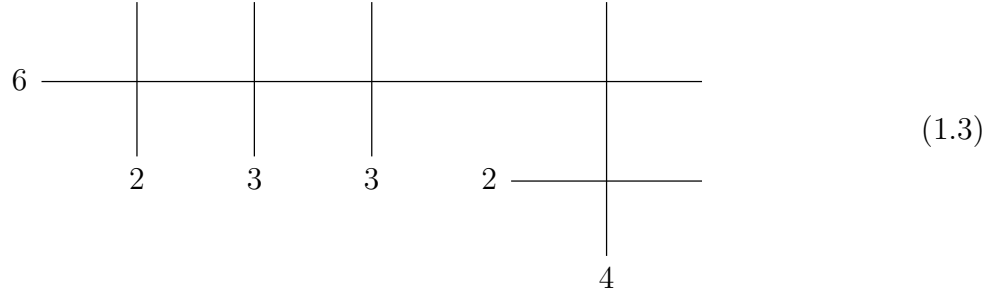
\begin{equation}\label{fig:1}
\begin{tikzcd}
                          & {} \arrow[dd, no head] & {} \arrow[dd, no head] & {} \arrow[dd, no head] &                       & {} \arrow[ddd, no head] &    \\
6 \arrow[rrrrrr, no head] &                        &                        &                        &                       &                         & {} \\
                          &   2                    & 3                      & 3                      & 2 \arrow[rr, no head] &                         & {} \\
                          &                        &                        &                        &                       & 4                       &   
\end{tikzcd}
\end{equation}
For any closed point $x \in X_s$ we have $\deg_k(x)\calN(x) \subset \calN(x) \subset 2\NN \cup 3\NN$, because the components containing $x$ either all have multiplicity divisible by $2$ or all have multiplicity divisible by $3$. 
All components are genus \(0\) curves, so there are \(k\)-points on the multiplicity \(2\) and \(3\) components that do not lie on any other component. For these we have $\deg_k(x)\calN(x) = 2\NN$ and $\deg_k(x)\calN(x) = 3\NN$. So if $K$ is Henselian, Theorem~\ref{thm:MainThm} gives $\calD(C/K) = 2\NN \cup 3\NN$. Then $C/K$ has index $1 = \GCD(\calD(C/K))$ by Corollary~\ref{cor:index}, and the degree set excludes all integers coprime to $6$. For more examples of similar phenomena see Proposition~\ref{prop:examples}.

		The possibilities for the (geometric) special fiber of a minimal proper regular model of a smooth geometrically irreducible curve of genus $2$ have been classified in \cites{Ogg,NU} (excluding some small residue characteristics). Using this classification we obtain the following.
		
\begin{thm}\label{thm:g2}
	Let \(K\) be a Henselian field with perfect residue field \(k\) and let \(C\) be a smooth projective geometrically irreducible genus \(2\) curve over \(K\).  If \(\Char(k) \nmid 30 \), then 
\[
	\calD(C/K) \in \left\{2\NN, \NN, \NN_{>1}, 2\NN\cup 3\NN\right\} \cup 
		\left\{
			 \cup_{d \in\calD(C_0/k)} d\NN : C_0/k \textup{ regular genus }2
		\right\}\,.
\]
 
\end{thm}
\begin{cor}\label{cor:QpAndQpnr}
	Let \(K\) be a Henselian field with residue field \(k\) and let \(C\) be a smooth projective geometrically irreducible genus \(2\) curve over \(K\).  Assume \(\Char(k) \nmid 30 \).
	\begin{enumerate}
		\item If $k$ is algebraically closed, then \(\calD(C/K) \in \left\{\NN, \NN_{>1}, 2\NN\cup 3\NN\right\}\).
		\item If \(k\) is finite, \(\calD(C/K) \in \left\{2\NN, \NN, \NN_{>1}, 2\NN\cup 3\NN\right\}\).
	\end{enumerate}
	In both cases, all of the possible degree sets are realized by a curve over $K$.
\end{cor}

		For fixed genus $g > 2$ we do not have a complete classification of the possible special fibers, but it is known that they fall into finitely many families whose combinatorial data can be bounded in terms of the genus (See \cite[1.6]{AW} and~\cite[10.1.57]{Liu}). Using this we prove the following. 
		
		\begin{thm}\label{thm:bounded}
			Suppose that $K$ is Henselian with {perfect} residue field $k$ that is either finite, {Hilbertian} or algebraically closed and let $g \ge 2$. There are only finitely many possibilities for the degree set $\calD(C/K)$ of a smooth projective and geometrically irreducible curve $C/K$ of genus $g$.
		\end{thm}
		
		We prove the slightly stronger Theorem~\ref{thm:FinitenessbbD} in the text, which gives the same conclusion for a larger class of field including, for example, fields such as $\Q_p((t_1))((t_2)) \cdots ((t_\ell))$.

\section*{Acknowledgements}
The first author was partially supported by the Marsden Fund Council administered by the Royal Society of New Zealand. The second author was partially supported by NSF DMS-2101434 and the AMS Birman Fellowship.  Additionally, this material is based partially upon work that was supported by National Science Foundation grant DMS-1928930 while the second author was in residence at the Simons Laufer Mathematical Sciences Institute in Berkeley, California, during the Spring 2023 semester.

\section{The Proof of Theorem~\ref{thm:MainThm}}

Throughout this section we assume that $X$ is a regular scheme, $X \to S = \Spec(R)$ is a proper, flat morphism, and that the generic fiber $X_\eta$ is geometrically irreducible. We use $R^h$ to denote the Henselization of $R$ with respect to the maximal ideal $(\pi) \subset R$ and $R\{t_1,\dots,t_d\}$ to denote the Henselization of the polynomial ring $R[t_1,\dots,t_d]$ with respect to the ideal $(t_1,\dots,t_d)$. Then $R^h\{t_1,\dots,t_d\}$ is canonically identified with the Henselization of $R[t_1,\dots,t_d]$ with respect to the maximal ideal $(\pi,t_1,\dots,t_d)$ by \cite[\href{https://stacks.math.columbia.edu/tag/08HT}{Tag 08HT}]{stacks-project}.

\subsection{The Henselization of the local ring of an SNC-point}

\begin{defn}
	Let $T \to S = \Spec(R)$ be a local scheme with closed point $t \in T$ and let $n_1 \ge \dots \ge n_d \ge 0$ be integers. We call $T$ an \defi{SNC-point with parameters $n_1,\dots,n_d$} if there exist a uniformizer $\pi \in R$, a unit $u \in R^h\{t_1,\dots,t_d\}^\times$ and an isomorphism of local $R$-algebras
	\[
		\calO_{T,t}^h \simeq R^h\{t_1,\dots,t_d\}/(t_1^{n_1}\cdots t_d^{n_d} - u\pi)\,,
	\]
	where $\calO_{T,t}^h$ denotes the Henselization of the local ring of $T$ at $t$.
\end{defn}

The terminology is justified by the following lemma.

\begin{lemma}\label{lem:SNCpoint}
	Suppose $X$ has dimension $d$ and $x \in X_s$ is a closed point with residue field $\kk(x) \simeq \rf$ such that the divisor $X_s \subset X$ has strict normal crossings at $x$ (See~\cite{Liu}*{Definition 9.1.6}). Let $n_1 \ge \dots \ge n_d \ge 0$ denote the multiplicities of the irreducible components of $X_s$ which contain $x$, where we take $n_{i} = 0$ for $i$ greater than the number of irreducible components containing $x$. Then $\Spec(\calO_{X,x}) \to S$ is an SNC-point with parameters $n_1,\dots,n_d$.
\end{lemma}

\begin{proof}
	Since the statement concerns the local ring of $X$ at $x$ we may replace $X$ by an affine subscheme containing $x$ as needed. Then we can find an embedding $X \to Y = \A^n_{S}$ for some $n$. Let $y \in Y$ denote the image of $x \in X$ and let $\mathfrak{m}_y \subset \calO_{Y,y}$ and $\mathfrak{m}_x \subset \calO_{X,x}$ denote maximal ideals in the corresponding local rings. Let $\overline{\mathfrak{m}}_y := \mathfrak{m}_y/(\pi) \subset \calO_{Y_s,y}$ and $\overline{\mathfrak{m}}_x := \mathfrak{m}_x/(\pi) \subset \calO_{X_s,x}$ denote the maximal ideals corresponding to $y$ and $x$ viewed as closed points of the special fibers. The maps on cotangent spaces induced by $X \to Y$ and $X_s \to Y_s$ yield a commutative diagram of ${\rf} = \kk(x) = \kk(y)$-vector spaces
		\begin{equation}\label{diag:cotangent}
			\xymatrix{
				0 \ar[r] & I/(I \cap \mathfrak{m}_y^2) \ar[r] & \mathfrak{m}_y/\mathfrak{m}_{y}^2 \ar[r] \ar@{->>}[d]  & \mathfrak{m}_x/\mathfrak{m}_x^2 \ar@{->>}[d] \ar[r]& 0\\
				&& \overline{\mathfrak{m}}_y/\overline{\mathfrak{m}}_{y}^2 \ar[r]  & \overline{\mathfrak{m}}_x/\overline{\mathfrak{m}}_x^2 \ar[r] & 0
				}
		\end{equation}
		where $I \subset \calO_Y(Y)$ is the ideal defining $X \subset Y$. We have $\dim_{\kk(y)}\mathfrak{m}_y/\mathfrak{m}_{y}^2 = \dim(Y) = n+1$, $\dim_{\kk(y)}\overline{\mathfrak{m}}_y/\overline{\mathfrak{m}}_{y}^2 = \dim(Y_s) = n$ and $\dim_{\kk(x)}\mathfrak{m}_x/\mathfrak{m}_{x}^2 = \dim(X) = d$, since $Y$, $Y_s$ and $X$ are regular. Counting dimensions we conclude that there exist $f_{d+1},\dots,f_n \in I$ whose images in $\overline{\mathfrak{m}}_y/\overline{\mathfrak{m}}_{y}^2$ are linearly independent.
		
		Since $X_s = \divv(\pi) \subset X$ has strict normal crossings at $x$, there is a system of parameters $g_1,\dots,g_d \in \calO_{X,x}$ and integers $n_1 \ge \cdots \ge n_d \ge 0$ such that $\ord_x(g_1^{n_1}\dots g_d^{n_d}) = \ord_x(\pi)$. Let $f_1,\dots,f_d \in \calO_{Y,y}$ be lifts of $g_1,\dots,g_d$. If $X_s$ is not regular at $x$, then
		\[
			{\dim X - 1=\dim X_s < \dim \overline{\mathfrak{m}}_x/\overline{\mathfrak{m}}_x^2 \le \dim X = \dim \mathfrak{m}_x/\mathfrak{m}_x^2}
		\]
		so the rightmost vertical map in~\eqref{diag:cotangent} is an isomorphism. It follows that $\pi,f_1,\dots,f_n$ is a system of parameters for $\calO_{Y,y}$. If \(X_s\) is regular at \(x\), then $\dim\overline{\mathfrak{m}}_x/\overline{\mathfrak{m}}_x^2 = d-1$, so we can find $f_1' \in I$ such that $f_1',f_{d+1},\dots,f_n$ have linearly independent images in $\overline{\mathfrak{m}}_y/\overline{\mathfrak{m}}_y^2$. Replacing $f_1$ with this $f_1'$ we again have that $\pi,f_1,\dots,f_n$ is a system of parameters for $\calO_{Y,y}$.
		
		Let $Z = V(f_{d+1},\dots,f_{n}) \subset Y$ and let $z$ denote the point above $y$. Then $Z$ is regular at $z$ and the images of $\pi,f_1,\dots,f_d$ under $\calO_{Y,y} \to \calO_{Z,z}$ (which we will denote with the same symbols) are a system of parameters for $Z$ at $z$. The condition defining the $g_i$ implies that $g_1^{n_1} \cdots g_d^{n_d} = w\pi$ for some unit $w \in \calO_{X,x}^\times$. Thus, in the case that $X_s$ is not regular at $x$ there is a unit $v \in \calO_{Z,z}^\times$ such that the kernel of the surjective map $\calO_{Z,z} \to \calO_{X,x}$ contains $F := f_1^{n_1}\cdots f_d^{n_d} - v\pi$. In the case that $X_s$ is regular at $x$ the kernel contains $F := f_1$. In either case, the morphism of regular local rings $\calO_{Z,z}/(F) \to \calO_{X,x}$ must be an isomorphism as it induces an isomorphism on the tangent spaces.
		
		As $Y = \A^n_S$, the $R$-module $\calO_{Y,y}$ is isomorphic to the localization of the polynomial ring $R[t_1,\dots,t_n]$ at the ideal $\mathfrak{m} = (\pi,t_1,\dots,t_n)$. Under this isomorphism the $f_i$ may be viewed as rational functions in the $t_j$. Note that $\{\pi,t_1,\dots,t_n\}$ and $\{ \pi,f_1,\dots,f_n\}$ generate the same ideal, since the latter were shown above to be a system of parameters for $Y$ at $y$. Thus we have a local $R$-module morphism (i.e., a morphism of local rings that is also a morphism of $R$-modules)
		\[
			R[t_1,\dots,t_n]_{(\mathfrak{m})} \to R[t_1,\dots,t_n]_{(\mathfrak{m})}\,,\quad t_i \mapsto f_i(t_1,\dots,t_n)\,,
		\]
		This induces a local $R$-module morphism
		\[
			\psi: R[t_1,\dots,t_d]_{(\mathfrak{m}')} \simeq \frac{R[t_1,\dots,t_n]_{(\mathfrak{m})}}{(t_{d+1},\dots,t_n)} \to \frac{R[t_1,\dots,t_n]_{(\mathfrak{m})}}{(f_{d+1},\dots,f_n)} \simeq \calO_{Z,z}\,,
		\]
		where $\mathfrak{m}' = (\pi,t_1,\dots,t_d) \subset R[t_1,\dots,t_d]$. Since $\pi,f_1,\dots,f_n$ generate $\mathfrak{m}$, $\psi$ is \'etale. Therefore, the Henselization of $\psi$ admits a section
		\[
			\sigma : \calO_{Z,z}^h \to R^h\{t_1,\dots,t_d\}\,.
		\]
		As these are regular local rings of the same dimension, this implies that $\sigma$ is an isomorphism (whose inverse is the Henselization of $\psi$). Composing with the Henselization of the isomorphism $\calO_{X,x} \simeq \calO_{Z,z}/(F)$ above we obtain an isomorphism
		\[
			\calO_{X,x}^h \simeq \calO_{Z/z}^h/(F) \simeq R^h\{t_1,\dots,t_d\}/(\sigma(F))\,.
		\]
		When $X_s$ is not regular at $x$ we have $(\sigma(F)) = ( t_1^{n_1}\cdots t_d^{n_d} - u\pi)$ with $u = \sigma(v) \in R^h\{t_1,\dots,t_d\}^\times$. When $X_s$ is regular at $x$ we have $\sigma(F) = t_1$ and so we obtain 
		\[
			R^h\{t_1,\dots,t_d\}/(\sigma(F)) = R^h\{t_2,\dots,t_d\} \simeq R^h\{t_1,\dots,t_d\}/ ( t_1^{n_1}\cdots t_d^{n_d} - u\pi)\,.\,
		\]
		since in this case $n_1 = 1$, $n_2 = \cdots = n_d = 0$.
\end{proof}

\begin{lemma}\label{lem:SNCn}
	Let $T \to S$ be an SNC-point with parameters $n_1,\dots,n_d$. Let $n$ be a positive integer. Consider the following statements.
		\begin{enumerate}
			\item\label{it:e=n} There exists a totally ramified extension $R'/R$ of degree $n$ with $T(R') \ne \emptyset$.
			\item\label{it:e} There exists an extension $R'/R$ of ramification index $n$ with $T(R') \ne \emptyset$.
			\item\label{it:n} There exists an extension $R'/R$ of degree $n$ with $T(R') \ne \emptyset$.
			\item\label{it:sum} There are positive integers $a_1,\dots,a_d$ such that $n = \sum a_in_i$. 
		\end{enumerate}
		Then~\eqref{it:e=n} implies~\eqref{it:e} and~\eqref{it:n}, each of which implies~\eqref{it:sum}. If $R$ is Henselian, then all four statements are equivalent.
\end{lemma}

\begin{proof}
	It is clear that~\eqref{it:e=n} implies both~\eqref{it:e} and~\eqref{it:n}. We now prove that~\eqref{it:e} and~\eqref{it:n} imply~\eqref{it:sum}. Suppose that $R'/R$ is a finite extension with $T(R') \ne \emptyset$. By the universal property of Henselization we have $T(R') = \Hom_{R}(\calO_{T,t},R') = \Hom_{R}(\calO_{T,t}^h,{(R')^h})$. Since $T$ is an SNC-point with parameters $n_1,\dots,n_d$, this yields a local morphism of $R$-modules
	\[
		\phi \colon R^h\{t_1,\dots,t_d\}/(t_1^{n_1}\cdots t_d^{n_d} - u\pi) \to {(R')^h}\,.
	\]
	Let $s_i = \phi(t_i) \in {(R')^h}$. Then $\prod_{i = 1}^d s_i^{n_i} = v\pi$, where $v = \phi(u) \in {((R')^h)}^{\times}$. Using $\nu$ to denote the discrete valuation on ${(R')^h}$ we find 
	\begin{equation}\label{eq:sum}
		\sum_{i = 1}^d \nu(s_i)n_i = \nu(\pi) = e(R'/R)\,.
	\end{equation}
	Since $\phi$ is a local morphism, $\nu(s_i) > 0$ for all $i = 1,\dots,d$. So equation~\eqref{eq:sum} gives the implication~\eqref{it:e} $\Rightarrow$~\eqref{it:sum}. Note that multiplying both sides of~\eqref{eq:sum} by the positive integer $f = [R':R]/e(R'/R)$ we obtain
	\[
		\sum_{i=1}^d f\nu(s_i)n_i = [R':R]\,.
	\]
	So this also proves the implication~\eqref{it:n} $\Rightarrow$~\eqref{it:sum}.
	
	Now assume that $R$ is Henselian. To complete the proof it suffices to prove the implication~\eqref{it:sum} $\Rightarrow$~\eqref{it:e=n}. Suppose $n = \sum_{i=1}^d a_i n_i$ with positive integers $a_i$. The substitutions $t_i = z^{a_i}$ for $i = 1,\dots,d$ define a morphism
	\[
		\psi \colon \calO_{T,t}^h \simeq \frac{R\{t_1,\dots,t_d\}}{(t_1^{n_1}\cdots t_d^{n_d} - u\pi)} \to \frac{R\{z\}}{(z^n - v \pi)}\,,
	\]
	where $v \in R\{z\}^\times$ is some unit.
	
	By the Weierstrass preparation theorem (which holds in this Henselian context by \cite{Moret-Bailly}*{Theorem 1.1(c)} or \cite[3.1.2]{BC}) there is a unit $w \in R\{z\}^\times$ and a polynomial $g(z) \in R[z]$ with coefficients in the maximal ideal of $R$ and degree less than $n$ such that $w(z^n - v \pi) = z^n + g(z)$. Evaluating at $z = 0$ we see that $g(0) = w(0)v(0)\pi$ has valuation $1$. Thus, the polynomial $f(z) := z^n + g(z)$ is an Eisenstein polynomial and so $R' := R[z]/(f(z))$ is a totally ramified extension of degree $n$. This yields the $R$-module morphism $R\{z\}/(z^n - v\pi) = R\{z\}/(f(z)) \to R'$. Composing with $\psi$ we obtain the desired morphism in $\Hom_R(\calO_{T,t}^h,R') = T(R')$.
\end{proof}

\subsection{When $R$ is Henselian}\label{sec:RHens}

If $R$ is Henselian, then there is a reduction map $r : X_\eta^0 \to X_s$, where $X_\eta^0$ denotes the set of closed points of the generic fiber $X_\eta$. For a closed point $P \in X_\eta$, the Zariski closure $\overline{P}$ of $P$ in $X$ is a finite irreducible scheme over $S$. The reduction of $P$ is the closed point $r(P) = \overline{P} \cap X_s \in X_s$ (cf. \cite{Liu}*{Definition 10.1.31}). Given a closed point $x \in X_s$ and a finite extension $L/K$ let us use $X_{\eta}(L)_x$ to denote the set of morphisms in $X_\eta(L)$ whose image in $X_\eta^0$ is a closed point reducing to $x$. If $R_L$ is the integral closure of $R$ in $L$, then properness of $X \to S$ implies that $X_\eta(L)_x \ne \emptyset$ if and only if $\Hom_R(\calO_{X,x},R_L) \ne \emptyset$. 

When $R$ is Henselian we define, for a closed point $x \in X_s$,
		\[
			\calD(x) := \left\{ [L:K] \::\; X_\eta(L)_x \ne \emptyset \right\} =  \left\{ [L:K] \::\; \Hom_R(\calO_{X,x},R_L) \ne \emptyset \right\}\,.
		\]
	\begin{lemma}\label{lem:calD(X)}
		Assume $R$ is Henselian. Then $\calD(X_\eta/K) = \bigcup_{x \in X_s} \calD(x)$.
	\end{lemma}	
	
	\begin{proof}
		If $R$ is Henselian, then $K$ is a large field \cite{Pop}. It follows from~\cite{LiuLorenzini}*{Proposition 8.3} that for any finite extension $L/K$ with $X_\eta(L) \ne \emptyset$, there exists a closed point $P\in X_\eta$ with residue field $L$. Thus
	\[
		\bigcup_{x \in X_s} \calD(x) \subset \calD(X_\eta/K)\,.
	\]
		The reverse containment holds when $R$ is Henselian because every closed point $P \in X_\eta$ reduces to some closed point $x \in X_s$ {(as opposed to a reducible effective \(0\)-cycle)}.
	\end{proof}
	\begin{prop}\label{prop:D(x)}
		Assume $R$ is Henselian. If $x \in X_s$ is a closed point such that $X_s$ has strict normal crossings at $x$, then $\calD(x) = \deg_\rf(x)\calN(x)$. 
	\end{prop}
	
	\begin{proof}
		First we show that $\calD(x)$ contains $\deg_k(x)\calN(x)$. Since $k$ is perfect, there is a unique unramified extension $K'/K$ of degree $\deg_k(x)$ corresponding to the residue field extension $\kk(x)/k$. Use $R'$ and $k' \simeq \kk(x)$ to denote the corresponding ring of integers and residue field, and write \(S' := \Spec(R')\). Then $S' \to S$ is \'etale, so the base change $X' = X \times_S S'$ is regular and $X' \to S'$ is a proper, flat morphism with geometrically irreducible generic fiber. Moreover, the special fiber of $X'$ contains a $k'$-point $x'$ above $x$. By Lemma~\ref{lem:SNCpoint}, $x' \in X'$ is an SNC-point with parameters equal to the multiplicities of the irreducible components passing through $x'$. Note that these are the same as the multiplicities of the irreducible components of $X_s$ passing through $x$, so $\calN(x') = \calN(x)$. By the equivalence of~\eqref{it:n} and~\eqref{it:sum} in Lemma~\ref{lem:SNCn} (applied to $x' \in X'$ with $k'$ in place of $k$) we have that $\calD(x') = \calN(x') = \calN(x)$. A closed point $P' \in X'_{K'}$ reducing to $x'$ yields a closed point $P \in X$ of degree $$\deg_K(P) = \deg_{K'}(P')[K':K] =  \deg_{K'}(P')\deg_k(x)$$ reducing to $x$. Thus we conclude that $\deg_k(x)\calN(x) \subset \calD(x)$.
		
	Conversely, suppose $L/K$ is a finite extension with $X_\eta(L)_x \ne \emptyset$. Let $K \subset K' \subset L$ be the maximal unramified subextension, {and let} $R'$ and $k'$ denote the corresponding valuation ring and residue field. Then the special fiber of $X' = X \times_S \Spec(R')$ contains a $k'$-point $x'$ above $x$ with $X'_\eta(L)_{x'} \ne \emptyset$. The equivalence of~\eqref{it:e} and~\eqref{it:sum} in Lemma~\ref{lem:SNCn}  implies that $e(L/K') = e(L/K) \in \calN(x') = \calN(x)$. Moreover, $\deg_{\kk(x)}(x')$ divides $[K':K]$, so $[L:K] \in \deg_k(x)\calN(x)$.
\end{proof}

\begin{proof}[Proof of Theorem~\ref{thm:MainThm}]
	This follows immediately from Lemma~\ref{lem:calD(X)} and Proposition~\ref{prop:D(x)}.
\end{proof}

\section{When $X_s$ does not have strict normal crossings}

	In this section we show how the set $\calD(x)$ introduced in Section~\ref{sec:RHens} can be described when $X_s$ does not have strict normal crossings at $x$ by using ideas in~\cite{GLL}*{Section 8}. Note that we do not rely on any moving lemmas from~\cite{GLL}, nor does the material in Sections~\ref{sec:DegCurves} and~\ref{sec:Genus2} rely on the material in this section. 
	
	As before we assume that $X$ is a regular scheme, $X \to S = \Spec(R)$ is a proper, flat morphism, and that the generic fiber $X_\eta$ is geometrically irreducible. For a closed point $x \in X_s$, the blow up of $\tilde{X} \to X$ at $x$ is a regular scheme with a proper flat morphism $\tilde{X} \to S$ and $\tilde{X}_\eta \simeq X_\eta$. By \cite{GLL}*{Proposition 8.3(1)}, the exceptional divisor of $\tilde{X} \to X$ is an irreducible component of the special fiber of $\tilde{X}$ isomorphic to $\PP_{\kk(x)}^{\dim X_s}$ with multiplicity equal to $\sum_{x \in E_i} m_ie(E_i,x)$, where $e(E_i,x)$ denotes the Hilbert-Samuel multiplicity of $x$ on $E_i$. Given a sequence 
	\begin{equation*}
		B_x \;:= \;\left( X_\ell \to X_{\ell-1} \to \cdots \to X_0 = X \right)
	\end{equation*}
	where each $X_i \to X_{i-1}$ is the blow up of a closed point $x_{i}$ in the special fiber of $X_{i-1}$ lying above $x$, let $e(B_x)$ denote the multiplicity of the exceptional divisor of $X_\ell \to X_{\ell-1}$ and let $d(B_x) = \deg_{\kk(x)}(x_\ell)e(B_x)$. Define 
	\[\calM(x) := \bigcup_{B_x} d(B_x)\NN \quad \text{and} \quad \calM'(x) := \bigcup_{B_x}e(B_x)\NN,\]
	where the unions run over all finite sequences of blow ups $B_x$ as above. By induction using the quoted formula for the multiplicity of an exceptional divisor we see that
	\[
		\calM(x) \subset \calM'(x) \subset \calN(x)\,.
	\]
	\begin{prop}\label{prop:D(x)M}
		Suppose $R$ is Henselian and $x \in X_s$ is a closed point. Then $$\calD(x) \subset \deg_\rf(x)\calM(x),$$ with equality if the residue field is infinite.
	\end{prop}

	\begin{proof}
	
	Let $P \in X_\eta$ be a closed point with residue field $L$ such that $r(P) = x \in X_s$. For any sequence of blow ups $B_x = \left(X_\ell \to X_{\ell-1} \to \cdots \to X_0 = X\right)$ above $x$, the closure of $P$ in $X_\ell$ intersects $(X_\ell)_s$ in a closed point, because $R$ is assumed to be Henselian. By \cite[Proposition 8.3(3)]{GLL}, there is a sequence $B_x$ such that the closure of $P$ intersects $(X_\ell)_s$ in a closed point $y$ that is a regular point on $((X_\ell)_s)_{\red}$. This regular point must lie on a unique irreducible component $E$ of $(X_\ell)_s$. We can assume that $y$ lies on the exceptional divisor of $X_\ell \to X_{\ell-1}$, so that the multiplicity of $E$ is $e(B_x)$. Then $y$ is an SNC-point with parameters $n_1 = e(B_x)$, $n_2 = \cdots = n_d = 0$. By Lemma~\ref{lem:SNCn} we have $e(L/K) \in e(B_x)\NN$. Therefore
	\[
		\deg_k(x)e(B_x) \mid \deg_k(x)e(L/K) \mid \deg_k(y)e(L/K) \mid [L:K]\,.
	\]
	By definition $\calM(x)$ contains all multiples of $e(B_x)$, so the sequence of divisibility relations above shows that $[L:K] \in \deg_k(x)\calM(x)$. This proves that $\calD(x) \subset \deg_k(x)\calM(x)$.
	
	Now suppose $k$ is infinite and let $d \in \calM(x)$. It will suffice to show that the generic fiber of $X$ has a closed point $P$ of degree dividing $\deg_k(x)d$ such that $P$ reduces to $x$. By definition of $\calM(x)$ there is a sequence of blow ups $B_x$ above $x$ such that $d$ is a multiple of the integer $d(B_x)$. Let $E$ be the final exceptional divisor arising in the sequence $B_x$. Then $E \simeq \PP_{k'}^{\dim X_s}$ for some finite extension $k'/\kk(x)$ and $d(B_x) = [k':\kk(x)]e(B_x)$, where $e(B_x)$ is the multiplicity of $E$. Since $k$ is infinite, the $k'$-points on $E$ are Zariski dense. In particular $E$ must contain a $k'$-point which is a regular point on the reduced special fiber. By Lemma~\ref{lem:SNCn} the generic fiber has a closed point of degree
	\[
	[k':k]e(B_x) = \deg_k(x)[k':\kk(x)]e(B_x) = \deg_k(x)d(B_x) \mid \deg_k(x)d.
	\]
	Thus $\deg_k(x)\calM(x) \subset \calD(x)$.
\end{proof}

\begin{thm}\label{thm:general}
	Assume that $R$ is Henselian. Then
\[
	\bigcup_{x \in X_s^\textup{SNC}} \deg_{{\rf}}(x)\calN(x)\subset\calD(X_\eta/K) \subset \bigcup_{x \in X_s} \deg_{{\rf}}(x)\calM(x) \subset \bigcup_{x \in X_s} \deg_{{\rf}}(x)\calN(x)\,,
\]
where $X_s^\textup{SNC}$ denotes the set of points where $X_s$ has strict normal crossings. If the residue field $k$ is infinite, then $\calD(X_\eta/K) = \bigcup_{x \in X_s} \deg_{{\rf}}(x)\calM(x)$.
\end{thm}
\begin{proof}
	The first containment follows from Lemma~\ref{lem:calD(X)} and Proposition~\ref{prop:D(x)}. The middle containment (and the claim that it is an equality when $k$ is infinite) follow from Lemma~\ref{lem:calD(X)} and Proposition~\ref{prop:D(x)M}. The final containment follows from the observation made just before Proposition~\ref{prop:D(x)M} that $\calM(x) \subset \calN(x)$, for all $x \in X_s$.
\end{proof}

\subsection{Necessity of the hypotheses}\label{sec:hypotheses}

	Comparing Proposition~\ref{prop:D(x)} and Proposition~\ref{prop:D(x)M} we see that $\calN(x) = \calM(x)$ for any closed point $x \in X_s$ at which $X_s$ has strict normal crossings. The following example shows that, in general, any of the containments 
	\[
		 \calM(x) \subset \calM'(x) \subset \calN(x)
	\]	
	can be proper.
	
	\begin{example}\label{lem:exM}
		Suppose $X \to S$ is a relative curve and $E \subset X_s$ is an irreducible component of multiplicity $m$ with a simple node $x \in E(k)$ which does not lie on any other irreducible component of $X_s$. Then $\calN(x) = m\NN$, and $\calM'(x) = m\NN_{>1}$.
		\begin{enumerate}
			\item\label{it:tans} If the tangent directions to $E$ at $x$ are defined over $k$, then $\calD(x) = \calM(x) = m\NN_{>1}$.
			\item If the tangent directions to $E$ are conjugate over $k$, then $\calD(x) = \calM(x) = 2m\NN$.
		\end{enumerate}
	\end{example}
	
	\begin{proof}
		We have $\calN(x) = m\NN$ by definition. The Hilbert-Samuel multiplicity of $x$ on $E$ is $2$. The exceptional divisor $E'$ of the blow up of $X$ at $x$ has multiplicity $2m$ by \cite[Proposition 8.3(1)]{GLL} and meets the strict transform of $E$ transversally at either two distinct closed points with residue field $k$ or at a single closed point with residue field a quadratic extension of $k$, correspondingly as the tangent directions to $E$ at $x$ are or are not defined over $k$. The exceptional divisor of the subsequent blow up at these points  has multiplicity $3m$, and by induction one finds that $\calM'(x) = m\NN_{>1}$. In the case these points have residue field $k$, we have $\calM(x) = \calM'(x)$. In the case the residue field is a quadratic extension of $k$ we find $\calM(x) = 2m\NN$.
	\end{proof}

	The following example shows that Proposition~\ref{prop:D(x)} can fail if the strict normal crossings hypothesis is omitted, even if $R$ is Henselian and all irreducible components are regular.
	
	\begin{example}
		Suppose $X \to S$ is a relative curve whose special fiber is the union of two irreducible components $E_1,E_2$ which are regular and have multiplicities $m_1 = 2$ and $m_2 =  3$ and are such that the $E_1\cdot E_2 = 2x$ for some closed point $x \in X_s(k)$. Then $\calN(x) = \langle 2,3\rangle = \NN_{\ge 5}$, but $\calM(x) \subset 5\NN \cup \langle 2,10\rangle \NN \cup \langle 3,10\rangle\NN$, which is a proper subset of $\calN(x)$.
	\end{example}
	
	\begin{proof}
		$\calN(x) = \langle 2,3\rangle = \NN_{\ge 5}$ follows from the definition of $\calN(x)$. The exceptional divisor $E'$ of the blow up of $X$ at $x$ is a $\PP^1_k$ of multiplicity $5$, intersecting the strict transforms of $E_1$ and $E_2$ transversally at the same point $x' \in E'(k)$. The further blow up at $x'$ yields a special fiber consisting of an exceptional divisor or degree $10 = 2 + 3 + 5$ meeting the other components transversally at distinct points. The multiplicity of the exceptional divisor in any subsequent blow up at a closed point above $x$ must therefore be in one of the semigroups $\langle 5 \rangle$, $\langle 2, 10\rangle$ or $\langle 3, 10\rangle$.
	\end{proof}

	The next example shows that the containment $\calD(X_\eta/K) \subset \bigcup \deg_k(x)\calN(x)$ can fail when $R$ is not Henselian, even if $X_s$ is assumed to be a strict normal crossings divisor. 
	
	\begin{example}\label{ex:NotHens}
		Suppose $X \to S$ is the minimal proper regular model of the genus $2$ hyperelliptic curve considered in the introduction whose special fiber (which has strict normal crossings) is given by Figure~\eqref{fig:1}. Then $\bigcup_{x \in X_s} \deg_k(x)\calN(x) = 2\NN \cup 3\NN$. If $K$ is Hilbertian (e.g., $R = \Z_{(p)}$ and $K = \Q$), then $\calD(X_\eta/K) = \NN_{>1}$.
	\end{example}
	
	\begin{proof}
		It was shown in the introduction that $\bigcup_{x \in X_s} \deg_k(x)\calN(x) = 2\NN \cup 3\NN$. The assumption on $(\alpha^2-4) \in K^{\times 2}$ implies that the defining polynomial $\pi(x^6 + \alpha \pi x^3 + \pi^2)$ factors as a product of degree $3$ polynomials, giving points of odd degree on $X_\eta$. Hence, $X_\eta$ has index $1$. If $K$ is Hilbertian, then this implies that $\calD(X_\eta/K)$ is cofinite in $\NN$ by \cite[Proposition 7.5]{GLL}.
	\end{proof}
					
		\begin{rmk}
			If in the example above one instead takes $\alpha$ so that $\alpha^2 - 4\ \in R^\times \setminus R^{\times 2}$, then the two geometric components of multiplicity $3$ in Figure~\eqref{fig:1} will be Galois conjugate. In this case $X_s$ contains only one component of multiplicity $3$ (which is geometrically \emph{reducible}) intersecting the multiplicity $2$ component in a degree $2$ point, and so $2\NN \subset \deg_k(x)\calN(x)$ for all $x \in X_s$.
		\end{rmk}

\section{Degree sets of curves}\label{sec:DegCurves}

In this section, we prove Theorem~\ref{thm:bounded} from the introduction. We deduce this from a more general theorem (Theorem~\ref{thm:FinitenessbbD}) about the boundedness of degree sets, which we prove by combining Theorem~\ref{thm:MainThm} with results of~\cites{AW, Liu} that bound the possible combinatorial configurations of minimal regular models of genus \(g\) curves. 

%
%
%
%

\subsection{Finiteness of the possible degree sets for curves of fixed genus}\label{sec:thmFiniteness}

Let \(\calF\) be a set of fields and let \(g,\delta\in \Z\) with \(g\geq 0\) and \(\delta\geq 1\).  Let 
\(	\calC_{g,\delta,F}	\) denote the set of smooth proper geometrically integral curves of genus \(g\) with index \(\delta\) over \(F\).  We define
\[
	\mathbb{D}_{g,\delta,\calF} := \{\calD(C/F) : F\in \calF, C \in \calC_{g,\delta,F}\}
	\quad\textup{and}\quad
	\mathbb{D}_{g,\calF} := \bigcup_{\delta \ge 1} \mathbb{D}_{g,\delta,\calF}\;.
\]
If \(\calF = \{F\}\) then we write \(\mathbb{D}_{g,\delta,F}\) instead of \(\mathbb{D}_{g,\delta,\{F\}}\) and similarly for $\mathbb{D}_{g,\delta,F}$. Then $\mathbb{D}_{g,F}$ is the set of all degree sets of smooth curves of genus \(g\) over a the field $F$.
\begin{remarks}\label{rmks}\hfill
	\begin{enumerate}
		\item If \(\calD\in \mathbb{D}_{g, \calF}\), then \(\calD\) is the degree set of a curve whose index equals \(\gcd(\calD)\), so the union $\mathbb{D}_{g,\calF} = \bigcup \mathbb{D}_{g,\delta,\calF}$ is disjoint and this partition can be recovered from the elements.
		\item The sets $\mathbb{D}_{g,\delta,\calF}$ can only be nonempty when $\delta$ divides $2g-2$, the degree of the canonical divisor. If \(g\ne 1\), then \(2g-2\neq 0\), and so (for this fixed \(g\neq 1\))
		\(\mathbb{D}_{g, \calF}\) is finite if and only if \(\mathbb{D}_{g,\delta,\calF}\) is finite for all \(\delta\).  
		\item If \(g = 1\) there is no absolute bound on the index. For example, $\mathbb{D}_{1,\delta,\Q_p} \ne \emptyset$ for all $\delta \ge 1$~\cites{Clark, Sharif-Local} and so $\mathbb{D}_{1,\Q_p}$ is infinite.
 	\end{enumerate}
\end{remarks}

The following lemma shows that \(\mathbb{D}_{g, \delta, \{\textup{finite fields}\}}\) is finite, for any $g,\delta$.

\begin{lemma}\label{lem:finitefields}
	Fix a nonnegative integer \(g\).  There exists a positive integer \(r = r(g)\) such that for any finite field \(F\) and any smooth curve \(C/F\) of genus \(g\), the degree set \(\calD(C/F)\) contains \(\NN_{\geq r}\). Furthermore, if \(g=2\), then \( \bigcup_{d \in \calD(C/F)} d\NN \in \{ \NN, \NN_{>1} \} \,.\)
\end{lemma}
\begin{proof}
	Let \(C\) be a smooth projective curve over the finite field \(\F_q\) with $q$ elements. The set \(C(\F_{q^d}) \setminus \cup_{r|d, r\ne d} C(\F_{q^r})\) is nonempty whenever \(\#C(\F_{q^d}) > \sum_{p|d}\#C(\F_{q^{d/p}})\), where the sum ranges over prime divisors.
	By the Weil conjectures, this holds whenever
	\[
		q^d - \sum_{p|d\textup{ prime}}q^{d/p} - 2g\left(q^{d/2} + \sum_{p|d\textup{ prime}}q^{d/2p}\right) + 1 - \omega(d)>0\,,
	\]
	where $\omega(d)$ denotes the number of prime divisors of $d$. If this inequality is not satisfied then
		\[
		2g \ge 
		\frac{q^d - \sum_{p}q^{d/p} - (\omega(d)-1)}{q^{d/2} + \sum_{p}q^{d/2p}} \ge 
		\frac{q^d - \omega(d)q^{d/2} - (\omega(d)-1)}{q^{d/2}(\omega(d)+1)} \ge \frac{q^{d/2}}{\omega(d)+1}\, - 2 \ge \frac{2^{d/2}}{d}\, - 2
	\]
	One checks that $2^{d/2}/d > d-7$, so \( d \in \calD(C/\F_q)\) for all $d \ge 2g+9$.
	
	Now suppose $g = 2$. If $C(\F_q) \ne \emptyset$, then $\bigcup_{d \in \calD(C/\F_q)} d\NN = \NN$. So suppose $C(\F_q) = \emptyset$. Then $2 \in \calD(C/\F_q)$ because $C$ admits a degree $2$ map to $\PP^1$ and the preimage of any point in $\PP^1(\F_q)$ must be a degree $2$ point on $C$. By Riemann-Roch, $C$ admits an effective divisor of degree $3$. Since $C(\F_q) = \emptyset$, any such divisor must be irreducible. Thus, $3 \in \calD(C/\F_q)$. If $d > 3$, then $\#\F_{q^d} \ge 16$ in which case the Weil conjectures show that $C(\F_{q^d}) \ne \emptyset$. It follows that $\calD(C/\F_q)$ contains every prime integer and so $\bigcup_{d \in \calD(C/\F_q)} d\NN = \NN_{>1}$.
\end{proof}

\begin{remark}
	In the previous lemma we give all possibilities for the set $\bigcup_{d \in \calD(C/\F_q)} d\NN$ when $C$ is a genus $2$ curve over a finite field, because this is the relevant set for Theorem~\ref{thm:g2}, which we will prove in the following section. This does not give the full list of possibilities for the degree set of a genus $2$ curve over a finite field. For example, there are genus $2$ curves over small finite fields $\F_q$ such that $\calD(C/\F_q)$ contains $1$ but not $2$ (i.e., such that $C(\F_q) = C(\F_{q^2}) \ne \emptyset$)~\cite{LMFDB}*{\href{https://www.lmfdb.org/Variety/Abelian/Fq/2/2/b_a}{2.2.b\_a}}.
\end{remark}

\begin{thm}\label{thm:FinitenessbbD}
	Let \(\calF\) be a set of Henselian fields, let \(\calF_0\) be the set of residue fields arising from \(\calF\), and let \(g\geq0\) be an integer.  Assume that \(\calF_0\) is closed under taking finite extensions and that all fields in $\calF_0$ are perfect.  If \(\mathbb{D}_{g_0,\delta_0,\calF_0}\) is finite for all \(0\leq g_0\leq g\) and all \(\delta_0\geq 1\), then \(\mathbb{D}_{g,\delta,\calF}\) is finite.
\end{thm}

The proof of Theorem~\ref{thm:FinitenessbbD} is given below. First we use this to prove Theorem~\ref{thm:bounded} of the introduction.

\begin{proof}[Proof of Theorem~\ref{thm:bounded}]
	The set \(\mathbb{D}_{g_0, \delta_0, \{\textup{finite fields}\}}\) is finite by Lemma~\ref{lem:finitefields} and \(\mathbb{D}_{g_0, \delta_0, \{\textup{Hilbertian fields}\}}\) is finite by \cite[Prop. 7.5]{GLL}.
	For any algebraically closed $k$ we have \(\mathbb{D}_{g_0, k} = \{ \{ 1\} \}\)
	which is finite. So Theorem~\ref{thm:FinitenessbbD} implies that $\calD_{g,\delta,K}$ is finite for any Henselian field $K$ with finite, Hilbertian or algebraically closed residue field. For $g \ge 2$, this implies that $\mathbb{D}_{g,K}$ is finite by Remarks~\ref{rmks}.
\end{proof}

\begin{rmk}\label{rmk:iterated}
Theorem~\ref{thm:FinitenessbbD} can be applied inductively to prove that $\mathbb{D}_{g,K}$ is finite for $g \ge 2$ and $K$ an iterated Laurent series field $F((t_1))\cdots((t_\ell))$ over a field $F$ such as $\Q_p$ or $\Q(x_1,\dots,x_n)$.
\end{rmk}

\begin{proof}[Proof of Theorem~\ref{thm:FinitenessbbD}]
	Since any genus \(0\) curve can be embedded as a smooth conic, we have that \(\mathbb{D}_{0, \calF} \subset \{\NN, 2\NN\}\).  Further, the Riemann-Roch Theorem shows that if \(\calC_{1,\delta, K}\neq\emptyset\) then \(\mathbb{D}_{1,\delta,\calF}=\{\delta\NN\}\).
	Thus, we may assume that \(g\geq 2\).

	Let \(K\in \calF\), let \(R\) be the associated discrete valuation ring and let \(k\) denote its residue field. Let \(C\in \calC_{g,\delta, K}\). By~\cite{Liu}*{Proposition 9.3.36(b)}, there is a minimal SNC-model \(X \to \Spec(R)\) of \(C/K\). If $X_s$ is irreducible, then Theorem~\ref{thm:MainThm} implies that \(\calD(C/K) = \cup_{d\in \calD((X_s)_{\red}/k)}md\NN\).  By assumption, there are finitely many possibilities for \(\calD((X_s)_{\red}/k)\) and (noting that $X$ must be the minimal regular model of $C$ in this case)~\cite[1.6]{AW} shows that $m(E)$ is bounded by a constant depending only on the genus.  Thus, there are finitely many possibilities for \(\calD(C/K)\). Henceforth we assume $X_s$ is {reducible.}
	
	Let \(E_1, \ldots, E_n\) denote the irreducible components of \(X_s\), and let \(m_1, \dots, m_n\) denote their multiplicities in \(X_s\). For each $i = 1,\dots, n$, define $E_i^{\circ} := E_i - \cup_{j \neq i} (E_j \cap E_i)$ and $k_i :=\kk(E_i)\cap \kbar$. 
	
	Note that the number of possibilities for $\calN(x)$ with $x \in X_s$ can be bounded in terms of $M := \max(m_i)$. Let $B := \max\{\deg_k(x) \,:\, \textup{$x$ is a singular point on $(X_s)_{\red}$}\}$. Since $X_s$ is connected {and reducible}, every component contains a singular point of $(X_s)_{\red}$. Thus the degrees $[k_i:k]$ and the indices $\delta(E_i/k)$ are bounded by $B$.
	
	Given a subset $D \subset \NN_{\ge 1}$, the number of subsets of $D$ containing $D \cap \NN_{\ge B}$ can be bounded depending only on $D$ and $B$, as can the number of subsets of $D$ contained in $D \cap \NN_{\le B}$. For a given $i \in \{1,\dots,n\}$ the sets $\{ \deg_k(x) \;:\; x \in E_i^\circ \}$ and $\{ \deg_k(x) \;:\; x \in E_i \cap E_j \,, i \ne j\,\}$ determine such subsets of $D = \calD(E_i/k)$. It follows that, given $\calD(E_i/k)$, the number of possibilities for 
	\[
		\bigcup_{x\in E_i}\deg_k(x)\calN(x) =\bigcup_{x \in E_i^\circ} \deg_k(x)m_i\NN  \quad \cup \quad \bigcup_{i\ne j}\bigcup_{x \in E_i \cap E_j} \deg_k(x)(m_i\NN + m_j\NN)\,
	\]
	can be bounded in terms of $M$ and $B$. 
	
	If $Z_i/k_i$ is an irreducible component of $(E_i)_{k_i}$, then $\calD(E_i/k) = [k_i:k]\calD(Z_i/k_i)$. We have $\delta(Z_i/k_i) \le \delta(E_i/k) \le B$. Since the special fiber \(X_s\) is connected, the arithmetic genus of any geometric irreducible component of \(X_s\) is bounded by \(g\). The assumptions of the theorem therefore imply that the number of possibilities for $\calD(Z_i/k_i)$ can be bounded in terms of $g$. So the number of possibilities for $\calD(E_i/k) = [k_i:k]\calD(Z_i/k_i)$ can be bounded in terms of $g$ and $B$. 
	
	Putting all of this together (and applying Theorem~\ref{thm:MainThm}) we see that the number of possibilities for
	\[
		\calD(C/K) = \bigcup_{x \in X_s}\deg_k(x)\calN(x) = \bigcup_{i=1}^n\bigcup_{x\in E_i}\deg_k(x)\calN(x) 
	\]
	can be bounded in terms of $g$, $M$ and $B$. Therefore, finiteness of $\mathbb{D}_{g,\delta,\calF}$ follows from Proposition~\ref{prop:bounding} below.
\end{proof}

\begin{prop}\label{prop:bounding}
	Let $g \ge 2$ be an integer. There exists a constant $c(g)$ depending only on $g$ such that, for every smooth projective and geometrically irreducible curve $C/K$ over the field of fractions of a Henselian discrete valuation ring $R$ with perfect residue field $k$, there exists a regular scheme $X$ with a proper flat morphism $X \to \Spec(R)$ such that
	\begin{enumerate}
		\item the generic fiber of $X\to S:= \Spec(R)$ is isomorphic to $C/K$,
		\item the special fiber $X_s$ of $X\to S$ is a strict normal crossings divisor on $X$,
		\item the multiplicities of the irreducible components of $X_s$ are bounded by $c(g)$, and
		\item for every closed point $x\in X_s$ that is not regular on $(X_s)_{\red}$, we have $\deg_k(x) \le c(g)$.
	\end{enumerate}
\end{prop}

\begin{proof}
	Let $X$ be a regular scheme with a proper flat morphism $X \to S$ of relative dimension $1$. Consider the dual graph of the geometric special fiber of $X$. This is the graph whose vertices correspond to the geometric irreducible components (and are labelled with the multiplicity of the corresponding component) and with edges between two given vertices corresponding to the intersection points of the corresponding geometric components (and are labelled with the corresponding intersection multiplicity).

To begin, let us assume $k$ is algebraically closed and consider the effect on the dual graph of blowing up $X$ at a closed point $x$ on its special fiber. {Let $e(E_i,x)$ denote the Hilbert-Samuel multiplicity of $x$ on $E_i$}. The dual graph {of the special fiber of the blow up} will contain one additional vertex, corresponding to the exceptional divisor. The sum of the multiplicities of the edges incident to this new vertex is $\sum e(E_i,x)$, and the multiplicity is the sum $\sum m_i(E_i)e(E_i,x)$, where both sums range over the irreducible components. The multiplicities of the initial vertices are unchanged. For any two vertices both incident to the new vertex, the sum of the multiplicities of the edges connecting them (which is the intersection multiplicity of the corresponding components) must decrease (because we have blown up at a point of intersection of the corresponding components). All edges between two vertices of the initial graph not both incident to the new vertex remain unchanged and no other edges are introduced. 

As we now explain, there is a sequence of blow ups at geometric points whose length (i.e., the number of blow ups in the sequence) is bounded in terms of the sums
	\[
		\calI := \sum_{\substack{i\ne j\\E_i \cdot E_j > 1}} (E_i \cdot E_j)\,, 
		\quad
		\calE := \sum_{i}\sum_{\substack{x \in E_i\\ e(E_i,x) > 1}}e(E_i,x)\,,
		\quad \text{ and } \quad
		\calG := \sum_{\substack{i\\p_a(E_i)>0}}p_a(E_i)\,,
	\]
which results in a model for which all components of the special fiber are regular and their pairwise intersection numbers are at most $1$.
By \cite{Liu}*{Proof of Proposition 8.1.26}, if $x$ is a singular point on some irreducible component $E_i$, then $e(E_i,x) > 1$ and there is a sequence of {at most $p_a(E_i)$} 
blow ups for which all points in the fiber above $x$ are regular points of the components they lie on. Thus after at most {$\sum_i\sum_{x, e(E_i,x)>1}p_a(E_i) \leq \calE\cdot\calG$} blow ups we obtain a model with all components regular (so $e(E_i,x) = 1$ for all $i$ and $x$). In the resulting model the corresponding sum \(\calG'\) is bounded by the original sum \(\calG\).  Although the corresponding sum $\calI'$ of intersection numbers greater than $1$ may be larger than $\calI$, the discussion in the previous paragraph implies that
we can bound $\calI'$ in terms of the initial $\calI$ and $\calE$. For any pair of intersecting components, blowing up at a point of intersection will reduce the intersection multiplicity of their strict transforms. So after a further $\calI'$ blow ups we can obtain a model such that all intersection numbers are at most $1$.

Now consider the situation when $k$ is not assumed to be algebraically closed. There is an action of the absolute Galois group $\Gal(\kbar/k)$ on the dual graph induced by its action on the geometric special fiber $X_s \times_k \kbar$. The effect on the dual graph of blowing up $X$ at a closed point of the special fiber is the same as that of successively blowing up all points in the $\Gal(\kbar/k)$-orbit of a geometric point in its support. Note that in the preceding paragraph we are blowing up at a Galois stable set of points at every step (e.g., the set of geometric points that are not regular points on some irreducible component), so in fact we can obtain a sequence of blow ups of $X$ at closed points of $X_s$ resulting in a model where the geometric components are regular and have pairwise intersection numbers at most $1$. The total number of blow ups (weighted by degree of the centers) is bounded in terms of the data $\calI$, $\calE$, {and $\calG$} of the dual graph of the initial special fiber.

Now suppose $X \to S$ is the minimal regular model of $C$ over $R$. By~\cite{AW}*{Theorem 1.6} (see also~\cite{Liu}*{Prop. 10.1.57}) there is a constant $c_0(g)$ depending only on $g$ bounding the multiplicities $m(E_i)$. Moreover, op. cit. shows that there is a subgraph \(\Gamma\) of the dual graph of $X_s$ whose number of vertices and edges is likewise bounded by $c_0(g)$, {and such that each connected component in the complement $\Gamma^c$ is a path where $X_s$ has normal crossings for all points on a component corresponding to a vertex in this path}. In addition, {the number of connected components of $\Gamma^c$ is bounded by $c_0(g)$,} every vertex in \(\Gamma^c\) {corresponds to a component with $p_a = 0$ and} has at most one edge going to a vertex inside \(\Gamma\).  {In particular, the constant $c_0(g)$ can be taken to also bound the sums $\calI$, $\calE$, $\calG$, and to bound the size of the automorphism group of the dual graph.}

By the discussion {at the start of the proof} there is a sequence of blow ups {of $X$, whose length is bounded in terms of $\calI, \calE, \calG$, that results }in a model $X'$ with regular irreducible components and pairwise intersections of geometric components at most $1$. {Since $\calI, \calE, \calG$ are bounded by $c_0(g)$, we} can deduce bounds in terms of $c_0(g)$ for the multiplicities of the components of $X'_s$, the size of the automorphism group of its dual graph and the size of the subgraph $\Gamma'$ consisting of all edges and vertices incident to a vertex coming from the subgraph $\Gamma$. 

The special fiber of $X'$ has normal crossings {except at those points} that lie on more than two irreducible components. Blowing up once more at each of these points results in a model $X''$ with normal crossings. Each geometric point in the support of such a point gives a complete subgraph on at least 3 vertices (corresponding to the $\ge 3$ geometric components sharing that common intersection point) which is necessarily contained in $\Gamma'$. So the number of these can be bounded. We can therefore deduce a bound depending only on $g$ for the multiplicities of the components of $X''_s$ and the size of the automorphism group of its dual graph.

Note that $X''_s$ may not yet have \emph{strict} normal crossings (as there could be distinct geometric components in the same Galois orbit which intersect, resulting in an irreducible component with self intersections). The geometric points where this occurs correspond to edges in the dual graph which connect a pair of vertices in the same Galois orbit, and there is at most one such point in any connected component of $\Gamma'^c$.
Blowing up once at each of these points results in a model with strict normal crossings. Since the number of {such points} is bounded we can deduce a bound depending only on $g$ for the multiplicities of the irreducible components of the resulting special fiber and the size of the automorphism group of the resulting dual graph. 

Finally, we note that the closed points on the special fiber of this SNC model which are singular points on the reduced special fiber correspond to $\Gal(\kbar/k)$-orbits of edges in the dual graph. The size of these orbits, and hence the degrees of such points, are bounded by the size of the automorphism group of the dual graph, for which we have a bound depending only on $g$.
\end{proof}

\subsection{Further examples of degree sets not cofinite in $\delta\NN$}\label{sec:moreexamples}
Lemma \ref{lem:finitefields} shows that $\calD(C/F)$ is a cofinite subset of $\delta(C/F)\NN$ when $F$ is finite. The same is true when $F$ is Hilbertian by \cite[Prop. 7.5]{GLL}. Example~\eqref{ex:2Ncup3N} of the introduction shows that this does not hold over Henselian fields. The following proposition gives more examples illustrating this phenomenon.

\begin{prop}\label{prop:examples}
	Let $K$ be a Henselian field with finite residue field $k$. For any set of integers $n_1,\dots,n_\ell$ there exists a smooth hyperelliptic curve $C/K$ with degree set
	\[
		\calD(C/K) = 2\NN \cup n_1 \NN \cup \cdots \cup n_\ell \NN\,.
	\]
\end{prop}

\begin{proof}
	Let $f(x) \in R[x]$ be an even degree monic polynomial such that the reduction mod $\pi$ is a separable polynomial in $k[x]$ such that the set of degrees of the irreducible factors of $f(x)$ modulo $\pi$ is $\{ n_1,\dots,n_\ell \}$. Let $C/K$ be the hyperelliptic curve defined by $y^2 = \pi f(x)$. One can compute that the special fiber of the minimal regular model of $C$ over $R$ is geometrically given by a spine of multiplicity $2$ that transversely intersects $\deg(f(x))$ multiplicity $1$ components that are labeled by the roots of $f(x)$. The $\Gal(\kbar/k)$-action on the multiplicity $1$ components agrees with the Galois action of the roots of $f(x)$ modulo $\pi$, hence the irreducible components (over $k$) of multiplicity $1$ correspond to the irreducible factors of $f(x) \bmod \pi$, each consisting of $\Gal(\kbar/k)$-orbit of size $n_i$ of geometric irreducible components. For a closed point $x$ on such a component we have $\deg_k(x) \in n_i\NN$. Any point not on one of these multiplicity $1$ components lies on the multiplicity $2$ component and has $\calN(x) = 2\NN$. We conclude by Theorem~\ref{thm:MainThm}.
\end{proof}

\section{Degree sets of genus \(2\) curves}\label{sec:Genus2}

In this section we use the classification of minimal regular models of genus $2$ curves in \cites{Ogg,NU} to prove Theorem~\ref{thm:g2} and Corollary~\ref{cor:QpAndQpnr}.

\begin{lemma}\label{lem:LargeField}
	Let \(F\) be a large field and let \(C/F\) be a smooth genus \(2\) curve.  Then \(2\NN \cap \calD(\A^1/F) \subset \calD(C/F)\) and, if \(\delta(C) = 1\), then \(3\NN\cap \calD(\A^1/F) \subset \calD(C/F)\).
\end{lemma}

\begin{remark}\label{rmks2}\hfill
	Note that $\calD(C/F) \subset \calD(\A^1/F)$, since the latter is the set of degrees of finite extensions of $F$.
\end{remark}

\begin{proof}
	By~\cite{LiuLorenzini}*{Proposition 8.3}, if \(d\in \calD(C/F)\) then \(d\NN \cap \calD(\A^1/F) \subset \calD(C/F)\).  In particular, if \(C(F)\neq \emptyset\) then \(\calD(C/F) = \NN \cap \calD(\A^1/F)\).  So we may assume that \(C(F) = \emptyset\).  Then for any \(x\in \PP^1(F)\) the fiber above \(x\) in the hyperelliptic morphism \(C\to \PP^1\) gives a degree \(2\) point on \(C\).  If \(\delta(C) = 1\), then there is a \(k\)-rational \(0\)-cycle \(z\) of degree \(1\) and the Riemann-Roch theorem implies that \(3z\) is effective.  Since we assumed that \(C(F) = \emptyset\), \(3z\) must be linearly equivalent to an irreducible effective \(0\)-cycle, and thus \(3\in \calD(C/F)\), as desired.
\end{proof}

\begin{prop}[Consequence of~\cites{Ogg,NU}]\label{prop:NUSummary}
	Let \(K\) be the field of fractions of a Henselian discrete valuation ring $R$ with perfect residue field $k$ such that \(30\in R^{\times}\), let \(C/K\) be a smooth projective and geometrically irreducible genus \(2\) curve, and let \(X\to \Spec(R)\) be a minimal proper regular model of \(C\).

	Assume either that \(X_s\) contains a component of multiplicity coprime to \(6\) that has index coprime to \(6\) or that \(X_s\) has components \(E_1, \dots, E_r\) that meet in a point of degree coprime to \(6\) and such that \(\gcd(m(E_i))\) is coprime to \(6\).  Then at least one of the following holds:
		\begin{enumerate}
			\item \(X_s\) is a regular genus \(2\) curve over \(k\);\label{class:smoothgenus2}
			\item\label{it1} \(X_s\) is geometrically a union of two regular genus $1$ curves meeting transversely at a $k$-point;\label{class:2genus1}
			\item \(X_s\) has a regular geometrically irreducible component $E$ of multiplicity $1$ and geometric genus $1$ that meets the rest of $X_s$ at a $k$-point. On the blow up of $X$ at this point the exceptional divisor has multiplicity $3$ and meets the strict transform of $E$ at a $k$-point where the special fiber has strict normal crossings.\label{class:Genus1MeetingGenus0}
			\item\label{it2} \(X_s\) has a component \(E\) of multiplicity \(1\), index \(1\), geometric genus \(0\), and which contains an open \(U\) such that \(\deg(E-U)\leq 4\) and \(U\) is contained in the regular locus of \((X_s)_{\red}\);\label{class:genus0mult1} 
			\item\label{it3} \(X_s\) has components of multiplicity \(2\) and \(3\) that meet at a $k$-point where $X_s$ has strict normal crossings;\label{class:SNCMult2And3}
			\item\label{it4} \(X\) has Namikawa-Ueno type \textup{[II\(^*\) - IV - \(\alpha\)]} and \(\calD(C/K)\supset \NN_{>1}\).\label{class:3+4cup5}
		\end{enumerate}
\end{prop}
\begin{remark}
	The classification in~\cite{NU} is over an algebraically closed residue field (of characteristic $p \notin \{2,3,5\}$), so to determine the possibilities for the irreducible components and degrees of their points of intersection over $k$, one must consider the Galois action on the components. For example, the configuration depicted in Figure~\eqref{fig:1} admits an automorphism interchanging the multiplicity $3$ components. If this automorphism is realized by the Galois action on these geometric components (which is the case if $\alpha^2-4 \in R^\times\setminus R^{\times 2}$ in~\eqref{ex:2Ncup3N}), then the special fiber has a single irreducible component $E$ of multiplicity $3$ and for this component we have $[\kk(E) \cap \kbar : k] = 2$.
\end{remark}
\begin{lemma}\label{lem:NUexception}
	Let \(K\) be the field of fractions of a Henselian discrete valuation ring $R$ with perfect residue field, let \(C/K\) be a smooth projective and geometrically irreducible genus \(2\) curve, and let \(X\to \Spec(R)\) be a minimal proper regular model of \(C\).  Assume \(30\in R^{\times}\) and that \(X\) has Namikawa-Ueno type \textup{[II\(^*\) - IV - \(\alpha\)]}.  
	Then \(\calD(C/K)\supset \NN_{>1}\).
\end{lemma}
\begin{proof}
	In Namikawa-Ueno type \textup{[II\(^*\) - IV - \(\alpha\)]}, the special fiber contains components of multiplicity \(3\) and \(4\) that meet at a degree \(1\) SNC point and also a regular genus \(0\) component \(E\) of multiplicity \(5\) that meets the rest of \(X_s\) in a degree \(2\) subscheme.  Thus, Theorem~\ref{thm:MainThm} implies that \(\calD(C/K) \supset (3\NN + 4\NN)\cup (5\NN)\) and Lemma~\ref{lem:LargeField} implies that \(\calD(C/K)\) contains \(2\NN\cup 3\NN\).  Since \((3\NN + 4\NN)\cup 5\NN\cup 2\NN\cup 3\NN = \NN_{>1}\), we have proved the claim.
\end{proof}
\begin{proof}[Proof of Proposition~\ref{prop:NUSummary}] First we consider the case then \(X_s\) contains no components of multiplicity coprime to \(6\) that have index coprime to \(6\).  Then either the theorem vacuously holds or \(X_s\) must have components \(E_1, \dots, E_r\) such that \(\gcd(m(E_i))\) is coprime to \(6\).  In particular, \(X_s\) must have a component of multiplicity congruent to \(3\) modulo \(6\) and a component of multiplicity congruent to \(2\) or \(4\) modulo \(6\) that meet.  There are three types where this can occur:\footnote{Whether this does occur depends on the Galois action on the geometric components of the special fiber so it is not completely determined by Namikawa-Ueno type.}
	\begin{equation}\label{eq:TypesNoMultCoprimeTo6}
		[\textup{I}_0^*-\textup{IV}^*-\alpha], \quad
		[\textup{IV}^* - \textup{IV}^* -\alpha], \quad
		[\textup{III}^* - \textup{II}_0], 
	\end{equation}
	and one can check in all of these cases that~\eqref{class:SNCMult2And3} holds.

	Now assume that \(X_s\) has a component of multiplicity \(m\geq 2\) with \(\gcd(m, 6) = 1\).  There are twenty-four of these and (listed in the order they appear in~\cite{NU}) they are:
	\begin{center}
	\begin{tabular}{rrrrrr}
		[$\textup{V}^*$], &
		[$\textup{VII}^*$], &
	 	[$\textup{VIII}-2$], &
		[$\textup{VIII}-4$], & 
		[$\textup{IX}-3$], \\
		
		[$\textup{IX}-4$], &
		$[\textup{I}_0 - \textup{II}^* - m]$, & 
		$[\textup{I}_0^* - \textup{II}^* - m]$, & 
		$[\textup{I}_0^* - \textup{II}^* - \alpha]$, & 
		$[\textup{II} -\textup{II}^* - m]$,\\ 

		$[\textup{II}^* -\textup{II}^* - m]$, & 
		$[\textup{II}^*-\textup{II}^* -  	\alpha]$,&
		$[\textup{II}^* - \textup{IV} - m]$, &
		$[\textup{II}^* - \textup{IV} - \alpha]$, &
		$[\textup{II}^* - \textup{IV}^* - m]$, \\

		$[\textup{II}^* - \textup{IV}^* - \alpha]$, & 
		$[\textup{II}^* - \textup{III}-m]$, &
		$[\textup{II}^* - \textup{III} - \alpha]$,&
		$[\textup{II}^* - \textup{III}^* - m]$,&
		$[\textup{II}^* - \textup{III}^* - \alpha]$,\\

		$[\textup{II}^* - \textup{I}_n - m]$,&
		$[\textup{II}^* - \textup{I}_n^* - m]$,&
		$[\textup{II}^* - \textup{I}_n^* - \alpha]$,&
		$[\textup{II}^* - \textup{II}_n^*]$.
	\end{tabular}
\end{center}
Type  $[\textup{II}^* - \textup{IV}- \alpha]$ falls into class~\eqref{class:3+4cup5} by Lemma~\ref{lem:NUexception}. Types
\[
	[\textup{I}_0 - \textup{II}^* - 0], \; 
	[\textup{I}_0^* - \textup{II}^* -\alpha], \;
	[\textup{II}^* -\textup{II}^* - 0],\; 
	[\textup{II}^* - \textup{II}^* -\alpha], \;\textup{and} \; 
	[\textup{II}^* - \textup{IV} -\alpha], 
\] have multiplicity \(2\) and \(3\) components that meet in a unique point and that are fixed by every graph automorphism and so fall into class~\eqref{class:SNCMult2And3}.  (These six types together with~\eqref{eq:TypesNoMultCoprimeTo6} are the only Namikawa-Ueno types where the assumptions of the proposition can hold without the existence of a component in \(X_s\) with multiplicity \(1\) and index coprime to \(6\).)
The remaining twenty types contain a multiplicity \(1\) component \(E\) that is fixed by all automorphisms that fix another component whose multiplicity is coprime to \(6\), and further that this multiplicity \(1\) component 
has at most \(4\) points that are not regular points on $(X_s)_{\red}$. 

Thus, the only remaining case is when there are components of multiplicity \(1\) and there are no components of multiplicity \(m>1\) with \(m\) coprime to \(6\).  In addition to the case of smooth reduction (type \([\textup{I}_{0-0-0}]\)), which falls into class~\eqref{class:smoothgenus2}, the case of a chain of regular genus $1$ curves (type \([\textup{I}_{0}-\textup{I}_{0}-1]\)), which falls into class~\eqref{class:2genus1},
and three types ($[\textup{I}_0 - \textup{I}_0^*-0]$, 
$[\textup{I}_0 - \textup{IV} -0]$, and 
$[\textup{I}_0-\textup{IV}^*-0]$) 
which fall into class~\eqref{class:Genus1MeetingGenus0}, there are ninety other types that fall into this case, and in all of them we observe that there is a multiplicity \(1\) component \(E\) that is fixed by every graph automorphism with order coprime to \(6\), and furthermore that this component has geometric genus $0$ and has at most \(4\) points that are not regular points on $(X_s)_{\red}$.
\end{proof}




\begin{proof}[Proof of Theorem~\ref{thm:g2}]
	If \(\delta(C) = 2g-2 = 2\), then Lemma~\ref{lem:LargeField} implies that \(\calD(C/K) = 2\NN\).  If \(\delta(C) = 1\), then Lemma~\ref{lem:LargeField} implies that \(2\NN\cup 3\NN \subset \calD(C/K)\).  Let \(X\to \Spec(R)\) be a proper SNC model of \(C\).  If \(X_s\) is regular, then Theorem~\ref{thm:MainThm} implies that \(\calD(C/K) = \bigcup_{d\in \calD(X_s/k)}d\NN\).

	Now assume that \(X_s\) is not regular. We will prove that if \(\calD(C/K)\setminus(2\NN\cup 3\NN )\neq\emptyset\), then \({\NN_{>1}}\subset \calD(C/K)\) {which} will give our desired classification.  Assume that \(\calD(C/K)\) contains some \(d\) is coprime to \(6\), ie., \(d\notin 2\NN\cup 3\NN \).  Theorem~\ref{thm:MainThm} implies that there exists an \(x\in X_s\) such that \(d\in \deg_k(x)\calN(x)\).  In particular, \(\deg_k(x)\) must be coprime to \(6\) and \(\gcd(\{m(E) : x\in E\})\) is coprime to \(6\) (where \(E\) ranges over the irreducible components of \((X_s)_{\textup{red}}\)).  Note that for any birational morphism of proper regular varieties \(f\colon X \to X'\) and any curve \(E\ni x\) that is contracted by \(f\), \(\gcd(\{m(E') : f(x)\in E'\}) | m(E)\) and \(\deg(f(x)) | \deg(x)\).  Thus, \(\deg(f(x))\) and \(\gcd(\{m(E') : f(x)\in E'\})\) must be coprime to \(6\) as well. Since $X_s$ is not regular, we may assume the minimal regular model $X'$ satisfies one of the cases~\eqref{class:2genus1}-\eqref{class:3+4cup5} in Proposition~\ref{prop:NUSummary} (which follows from the Namikawa-Ueno classification).
	
	In~\eqref{class:2genus1}, \(X'_s\) is geometrically a union of two regular genus $1$ curves meeting transversely at a $k$-point.  If $X'_s$ is SNC at this \(k\)-point, then Proposition~\ref{prop:D(x)} implies that \(\NN\subset \calD(C/K)\). Otherwise we may blow up this point to obtain a new proper regular model \(X''\). Then \(X''_s\) contains, geometrically, a chain of three curves with transverse intersections, and the middle curve is a \(\PP^1\) over \(k\).  Since \(\#\PP^1(k)>2\) for any field \(k\), there must be a smooth \(k\)-point of \(X''_s\) on this component and so Proposition~\ref{prop:D(x)} again implies that \(\NN\subset \calD(C/K)\).
	
	In~\eqref{class:Genus1MeetingGenus0}, there is (after a blow up) an SNC $k$-point $x$ with $\calN(x) = \langle 1,3\rangle = \NN_{\ge 4}$. So \(\calD(C/K)\) contains \(\NN_{\ge 4}\) by Proposition~\ref{prop:D(x)}. Then $C/K$ has index is $1$, so $\NN_{\ge 1} \subset \calD(C/K)$ by Lemma~\ref{lem:LargeField}.

	In~\eqref{class:genus0mult1}, \(X'_s\) contains a \(k\)-rational component \(E\) with multiplicity \(1\), index $1$ and branch number \(b(E;X_s)\leq 4\). Since \(\Char(k)\nmid 30\), we have \(\#k\geq 4\). Thus, \(E\) contains a regular degree \(1\) point that does not lie on any other component of \(X_s\) so \(\NN\subset \calD(C/K)\).
	
	In~\eqref{class:SNCMult2And3}, \(X'_s\) contains a \(k\)-point that is the intersection of components of multiplicity \(2\) and \(3\) and \(X_s\) is SNC at this point.  Thus, by Proposition~\ref{prop:D(x)}, \(\calD(C/K)\supset 2\NN + 3\NN\).  Furthermore, Lemma~\ref{lem:LargeField} implies that \(\calD(C/K)\) contains \(2\NN\cup 3\NN\), and so \(\calD(C/K) \supset \NN_{>1}\).

	In the final case~\eqref{class:3+4cup5}, \(X\) has type \textup{[II\(^*\) - IV - \(\alpha\)]} and \(\calD(C/K)\supset \NN_{>1}\).
\end{proof}
\begin{rmk}
	In case~\eqref{class:genus0mult1} above we have made use of the assumption $\Char(k) \nmid 30$ to ensure the existence of sufficiently many $k$-points on some curve over $k$. With some additional effort one can avoid this and prove (without assumption on the residue characteristic) that for every genus $2$ curve $C/K$ whose geometric special fiber is one of the Namikawa-Ueno types, the degree set $\calD(C/K)$ is one of those listed in Theorem~\ref{thm:g2}.
\end{rmk}


	\begin{lemma}\label{lem:Realizable}
	Let \(K\) be a Henselian field with perfect residue field, let \(\pi\) denote a uniformizer, and let \(C/K\) be the smooth projective curve \(y^2 = \eps \pi(x^2 + \pi)(x^4 + \pi)\).  Then
	\[
		\calD(C/K) = \begin{cases}
			2\NN & \textup{if }\eps\notin K^{\times2}\\	
			\NN_{>1} & \textup{if }\eps\in K^{\times2}
		\end{cases}
	\]	
	\end{lemma}
	\begin{proof}
		This curve is type \([\textup{III}^* - \textup{II}_0]\) in the Namikawa-Ueno classification, and every component has multiplicity \(2, 3, \) or \(4\).  The dual graph of the special fiber has an order \(2\) automorphism which interchanges the two vertices corresponding to the irreducible components of multiplicity \(3\) (which do not meet each other) and has no other automorphisms. Thus, Theorem~\ref{thm:MainThm} implies that \(3\in \calD(C/K)\) if and only if \(\calD(C/K) = 2\NN \cup 3\NN \cup (2\NN + 3\NN) = \NN_{>1}\).  Furthermore, \(3\in \calD(C/K)\) if and only if there exists an \(x_0\in \overline{K}^{\times}\) with \(v(x_0)\in \frac13\Z\) and \(\eps\pi(x_0^2 + \pi)(x_0^4 + \pi)\in K(x_0)^{\times2}\).  One can check that this is possible if and only if \(\eps \in K^{\times2}\) (in that case any \(x_0\) with \(v(x_0) = \frac13\) will work).
	\end{proof}

	\begin{lemma}\label{lem:SHindex1}
		Any smooth projective geometrically integral hyperelliptic curve $C$ of even genus over a Henselian field $K$ with algebraically closed residue field has index $\delta(C/K) = 1$.
	\end{lemma}	
	\begin{proof}
		Let $X \to \Spec(R)$ be the minimal regular model of $C$ over the integer ring $R$ of $K$ and let $m_i$ be the multiplicities of the irreducible components of the special fiber. The adjunction formula implies that that $\GCD(m_i)$ divides $g(C)-1$~\cite[Exercise 9.1.8]{Liu}, so $\GCD(m_i)$ must be odd. Since $k$ is algebraically closed, Corollary~\ref{cor:index} implies that $\delta(C/K) = \GCD(m_i)$, so $\delta(C/K)$ is also odd. On the other hand, $\delta(C/K)$ divides $2$ since $C$ is hyperelliptic, and so $\delta(C/K) = 1$.
	\end{proof}

	\begin{remark}
		There exist hyperelliptic curves over strictly Henselian fields with odd genus that have index \(2\), e.g., \(y^2 = \pi(t^4 + \pi)(t^4 + u\pi)\) where \(\pi\) is a uniformizer and \(u\neq 1\) is a unit.
	\end{remark}	
	
\begin{proof}[Proof of Corollary~\ref{cor:QpAndQpnr}]
	From Theorem~\ref{thm:g2} and Lemma~\ref{lem:finitefields}, we deduce that if {$k$ is algebraically closed or finite}, then \(\calD(C/K) \in \left\{2\NN, \NN, \NN_{>1}, 2\NN\cup 3\NN\right\}\). Lemma~\ref{lem:SHindex1} shows that $2\NN$ does not occur when $k$ is algebraically closed.
	
	It only remains to show that all of the other possibilities do in fact occur. The curve given in~\eqref{ex:2Ncup3N} has degree set \(2\NN\cup 3\NN\). Any smooth projective genus \(2\) curve over \(K\) with a point (e.g., \(y^2 = x^5 - 1\) if the characteristic is different from \(2\) and \(5\)) has degree set equal to \(\NN\) by \cite{LiuLorenzini}*{Proposition 8.3}.  Finally, Lemma~\ref{lem:Realizable} shows that degree sets \(2\NN\) and \(\NN_{>1}\) are realizable if {$k$ is finite} and that the degree set \(\NN_{>1}\) is realizable if {$k$ is algebraically closed}.  
\end{proof}


\end{document}